\documentclass[10pt,reqno,twoside]{amsart}

\usepackage{enumerate}

\usepackage{amssymb}

\usepackage{amssymb,amsthm,amsmath,eucal,mathrsfs,verbatim}

\usepackage{footnote}

\usepackage{tikz, xcolor}

\usepackage{cite}

\usepackage{amsmath}

\usepackage{amsthm}

\usepackage{amssymb}

\usepackage{amsfonts}

\usepackage{hyperref}

\newtheorem{theorem}{Theorem}[section]

\newtheorem{lemma}[theorem]{Lemma}

\theoremstyle{definition}

\newtheorem{remark}[theorem]{Remark}

\numberwithin{equation}{section}

\providecommand{\norm}[1]{\lVert#1\rVert}

\newcommand\restr[2]{{
  \left.\kern-\nulldelimiterspace 
  #1 
  \vphantom{\big|} 
  \right|_{#2} 
  }}

\usepackage{eucal}

\newcommand{\wh}{\widehat}

\newcommand{\lb}{\left(}

\newcommand{\vp}{\varphi}
\newcommand{\ve}{\varepsilon}

\newcommand{\rb}{\right)}
\newcommand{\PD}{\partial}

\newcommand{\wt}{\widetilde}

\newcommand{\Rb}{\mathbb{R}}

\newcommand{\Beq}{\begin{equation}}
\newcommand{\Eeq}{\end{equation}}
\newcommand{\beq}{\begin{equation*}}
\newcommand{\eeq}{\end{equation*}}
\newcommand{\bal}{\begin{align}}
\newcommand{\eal}{\end{align}}
\renewcommand{\O}{\Omega}

\newcommand{\n}{\nabla}

\newcommand{\bpr}{\begin{proof}}
	\newcommand{\epr}{\end{proof}}
\renewcommand{\o}{\omega}

\newcommand{\tblue}[1]{{\color{blue}{#1}}}

\newcommand{\bel}[1]{\begin{equation}\label{#1}}
\newcommand{\ee}{\end{equation}}

\title[Increasing stability]{High frequency stability estimates for a partial data inverse problem}
\author[Choudhury and Krishnan]{Anupam Pal Choudhury$^{*}$ and Venkateswaran P. Krishnan$^{\sharp} $ } 
\date{}
\address{$^{*}$School of Mathematical Sciences, 
National Institute of Science Education and Research, Bhubaneswar 752050, India, and 
Homi Bhabha National Institute, Training School Complex, Anushaktinagar, Mumbai 400094, India. Email: {\tt anupampcmath@gmail.com} }
\address{$^{\sharp}$ TIFR Centre for Applicable Mathematics, Bangalore, India. Email: {\tt vkrishnan@tifrbng.res.in}}

\begin{document}

\begin{abstract}

In this article, high frequency stability estimates for the determination of the potential in the Schr\"odinger equation are studied when the boundary measurements are made on slightly more than half the boundary. The estimates reflect the increasing stability property with growing frequency.
\end{abstract}
\maketitle

\section{Introduction}
Let $\O\subset \Rb^n$ for $n\geq 3$ be a bounded domain with smooth boundary $\PD\O$. Consider the following boundary value problem in $\O$:
\begin{equation}\label{Elliptic PDE}
\mathcal{L}_{q}u:=\lb -\Delta -\o^2+ q \rb u=0,\quad 
u(x)=f(x) \ \text{on}\ \partial \Omega,
\end{equation}
where we assume, without loss of generality, that the real frequency $\o>1$ and $q\in H^{s}(\O)$ for an integer $s \ge \left[\frac{n}{2}\right]+1$.\\
In order to ensure uniqueness of solution of  the boundary value problem \eqref{Elliptic PDE}, following the works \cite{B,R}, we assume that
\begin{itemize}
\item[(A)]
$0$ is not a Dirichlet eigenvalue of $-\Delta +q $ in $\Omega $,
\end{itemize}
 and that the frequency $\omega $ is such that $0$ is not a Dirichlet eigenvalue of the operator $\mathcal{L}_{q} $ in $\Omega $. In particular, let $\Sigma_{q} $ denote the set of the inverse of eigenvalues of the operator $\left(-\Delta+q \right)^{-1} $. We assume that 
\begin{itemize}
\item[(B)]
$
 \text{dist}\left(\omega^{2},\Sigma_{q} \right) > c\ \omega^{2-n},\ \text{for some}\ c \ll 1.$
\end{itemize}
 For $M>0$ and $s$ as above, let us denote the admissible set of potentials: 
\Beq\label{Admissible Q}
Q_{M}:= \{q:  \norm{q}_{H^{s}(\Omega)} \leq M  \}. 
\Eeq

For $f\in H^{\frac{1}{2}}(\PD\O)$ in \eqref{Elliptic PDE}, let $u \in H^{1}(\Omega)$ be the unique solution of \eqref{Elliptic PDE}. The Dirichlet to Neumann map (DN)  $\Lambda_q: H^{\frac{1}{2}}(\PD\O) \to H^{-\frac{1}{2}}(\PD\O)$ is defined as 
$\partial_{\nu}u \vert_{\PD \O}$.  
We consider DN map restricted to certain open subsets of the boundary.  To precisely describe the set-up, let us introduce a few notation. 

Let $\nu$ denote the unit outer normal to $\PD\O$. Given a unit vector $\alpha \in \mathbb{S}^{n-1} $ and $\epsilon > 0 $, we define
\begin{align}
\notag&\partial \Omega_{+}:= \{x \in \partial \Omega,\ \alpha \cdot \nu(x)>0 \}, \ \partial \Omega_{-}:=\partial \Omega \setminus \overline{\partial \Omega}_{+}\ , \\
\label{Bd2}&\partial \Omega_{+,\epsilon}:= \{x \in \partial \Omega,\ \alpha \cdot \nu(x)>\epsilon \}, \ \partial \Omega_{-,\epsilon}:=\partial \Omega \setminus \overline{\partial \Omega}_{+,\epsilon}\ .
\end{align}
The partial DN map is defined  by $\wt{\Lambda}_q: H^{\frac{1}{2}}(\PD\O) \to H^{-\frac{1}{2}}(\PD\O_{-,\ve})$, $f\to \partial_{\nu}u \vert_{\PD \O_{-,\ve}}$. 

We are interested in the inverse problem of stable recovery of the potential $q$ from partial DN map $\wt{\Lambda}_{q}$. More precisely, we are interested in analyzing the stabilty estimates as $\o$ becomes large.

Corresponding to the case $\o=0$, the inverse problem of unique recovery of $q$ from the DN map $\Lambda_q$ began with the fundamental work of Calder\'on \cite{Calderon-InverseProblemPaper}, and was solved by Sylvester and Uhlmann in dimensions $n\geq 3$ in their groundbreaking work \cite{SU}. The stable recovery of the potential from the DN map was addressed by Alessandrini in \cite{Al} who showed that under an \emph{a priori} assumption of a uniform bound on the potentials, there is a stability estimate with a modulus of continuity of $\ln$ type. That such an estimate is optimal was shown by Mandache \cite{Ma}.

Again in the case $\o=0$, the unique recovery of the potential from partial DN map has received significant attention as well. The work \cite{BU} showed that one can uniquely recover the potential $q$ from the partial DN map $\wt{\Lambda}_q$ defined above. This work was signficantly improved in another fundamental work \cite{KSU}. Heck and Wang derived stability estimates of $\ln \ln$ type (see \cite{Heck-Wang-StabilityPaper}) for the recovery of $q$ from the partial DN map  $\wt{\Lambda}_q$ when the boundary measurements were made on slightly more than half the boundary, and stability estimates of $\ln$ type (see \cite{Heck-Wang-Optimal}) for partial data problems in certain special type of geometries \cite{Isakov-1}.

 For the full data case, the behavior of the stability estimates as the frequency grows was addressed by Isakov \cite{Isakov-2}. He showed that as the frequency $\o$ gets large, the logarithmic-type stability estimates for the full data case improves to Lipschitz-type stability estimates. For other closely related increasing stability works, we refer the reader to the following works of Isakov and his collaborators \cite{ILW,INUW, IW} and also \cite{Is1, Is2, Is-No1, Is-No2, NUW}. 

In the current work, we are interested in the question of analyzing the behavior of the stability estimates as the frequency $\o$ grows for the partial data inverse problem; the recovery of $q$ from $\wt{\Lambda}_q$. Recall from the work of Heck and Wang \cite{Heck-Wang-StabilityPaper} that the stability estimates are of $\ln \ln$ type.  We show that these estimates improve to Lipschitz-type stabililty estimates as the frequency $\o$ becomes large enough.

We would like to mention that the analysis of the behavior of the stability estimates from partial DN map as the frequency grows has been either known only in certain special type of geometries  (see \cite{CH, Liang}) or with impedance type boundary conditions under the assumption of knowledge of the potential in a neighbourhood of the boundary (see \cite{KU}) . In our work, we address this question for the partial data case considered by \cite{BU} and \cite{Heck-Wang-StabilityPaper}. 

%


To study the stability estimates in our set-up, following \cite{Heck-Wang-StabilityPaper}, we shall use a more regular Sobolev space. We shall assume that $f \in H^{\frac{3}{2}}(\partial \Omega) $ and hence the solution $u$ to \eqref{Elliptic PDE} is in $H^{2}(\Omega) $. The partial DN map $\tilde{\Lambda}_{q}$, therefore, now maps $H^{\frac{3}{2}}(\partial \Omega) $ to $H^{\frac{1}{2}}(\partial \Omega_{-,\epsilon}) $. 

We now state the main result of the paper.
\begin{theorem}\label{Main-result}
Let $\O\subset \Rb^n,\ n\geq 3$, be a bounded domain with smooth boundary $\PD \O$. 
Consider \eqref{Elliptic PDE} for two potentials $q_1$ and $q_2$ belonging to the admissible set \eqref{Admissible Q} and satisfying assumption $(A)$. Then there exist constants  $K>1,\ \theta \in (0,1),\ C=C(\O, n, M, K, \theta,\ve, s)$ and $\eta=\eta(s,n)$, such that for all $\omega> 1 $ satisfying the assumption \[\text{dist}\left(\omega^{2},\Sigma_{q_{i}} \right) > c\ \omega^{2-n}\ (\text{for some}\ c \ll 1),\ i=1,2,\] we have
\begin{equation}
\| q_{1}-q_{2}\|_{L^{\infty}(\Omega)}  \leq   C \left[\omega^7 \|\wt{\Lambda}_{q_{1}}-\wt{\Lambda}_{q_{2}} \|+\frac{1}{\left[\frac{1}{K} \ln \left(\ln \ \omega+\vert \ln\ \|\wt{\Lambda}_{q_{1}}-\wt{\Lambda}_{q_{2}} \| \vert \right) \right]^{\frac{2}{\theta}}} \right]^{\frac{\theta \eta}{2(1+s)}},
\label{stab-est}
\end{equation}
where $\wt{\Lambda}_{q_1}$ and $\wt{\Lambda}_{q_2}$ denote the partial DN maps (corresponding to $q_{1} $ and $q_{2} $ respectively) measured on the open subset $\PD \O_{-,\ve}\subset \PD\O$.
The constants $K,\theta,C $ and $\eta$ are independent of the frequency $\omega $.
\end{theorem}

The estimate \eqref{stab-est} clearly shows that as the frequency $\omega $ grows, the second term in the right-hand side decays to zero and the first term, which is the Lipschitz part, dominates. Thus the property of increasing stability is exhibited in this case.

In Section 2, we discuss some preliminary results that we shall need in the proof of the estimates. Section 3 deals with the proof of the stability estimate \eqref{stab-est}.

\section{Preliminaries}
In this section, we collect some preliminary results that will be used in the proof of Theorem \ref{Main-result}. \\
We begin with the derivation of the following boundary Carleman estimate. The proof closely follows \cite{BU} but the main point to note here is that the constants appearing in the estimate are independent of the frequency $\omega$.

\begin{theorem}\label{BCE}
Let $q$ in \eqref{Elliptic PDE} belong to $L^{\infty}(\Omega)$ and $\alpha$ be a unit vector in $\mathbb{R}^{n}$. Define $\varphi(x)=\alpha\cdot x$. Then there exist constants $C>0$ and $\lambda_{0}>0$ (both independent of $\omega$ and depending only on the domain $\Omega $ and $\|q \|_{L^{\infty}(\Omega)} $) such that for all $\lambda >\lambda_0$ and for all $u\in H^{2}(\Omega) \cap H^{1}_{0}(\Omega)$,
\begin{align}\label{carleman}
-\frac{1}{ \lambda } \int_{\partial \Omega_{-}} \left( \alpha \cdot \nu \right) \partial_{\nu}  u \  \overline{\partial_{\nu}u}\ dS+C  \|u \|^{2}_{L^{2}(\Omega)} & \leq \frac{1}{\lambda^2}\lVert e^{\lambda \vp}(\Delta +\o^2-q) e^{-\lambda \vp}u\rVert^2_{L^{2}(\O)} \\
\notag &\qquad +\frac{1}{\lambda} \int_{\partial \Omega_{+}} \left( \alpha \cdot \nu \right)  \partial_{\nu}  u \  \overline{\partial_{\nu}u} \ dS,
\end{align}


\end{theorem}
\begin{proof}
We prove the estimate for $C^{\infty}(\overline{\Omega}) $ such that $u=0 $ on $\partial \Omega $. The general case follows from a standard approximation argument.\\
Note that
\begin{align*}
\lVert e^{\lambda \vp} \lb -\Delta - \o^2\rb  e^{-\lambda \vp} u\lVert^2_{L^{2}(\Omega)} 
& =\lVert\lb  \Delta + \lambda^2+\o^2\rb u\rVert_{L^{2}(\Omega)}^2+4\lambda^2 \lVert \alpha \cdot \nabla u\rVert_{L^{2}(\Omega)}^2\\
&\underbrace{-2\lambda \langle \alpha \cdot \n u,\lb \Delta + \lambda^2+\o^2\rb u\rangle}_{I} 
\underbrace{-2\lambda \langle \lb \Delta+\lambda^2+\o^2\rb u, \alpha \cdot \n u \rangle}_{II}.
\end{align*}
Let us consider the third and fourth terms on the right from the above equation, and split it as 
\[
I+II=-2\lambda\lb  \underbrace{\langle \alpha \cdot \n u, \Delta u\rangle}_{A} + \underbrace{\langle \Delta u, \alpha \cdot \n u\rangle}_{B} + \underbrace{\langle \alpha \cdot \n u, \lb \lambda^2+\o^2\rb u\rangle}_{C}+ \underbrace{\langle \lb \lambda^2+\o^2\rb u,\alpha \cdot \n u\rangle}_{D} \rb.
\]
The third and the fourth expression on the right combine to give
\begin{align*}
-2\lambda \cdot \left(C+D\right)& =-2\lambda \lb \lambda^2+ \o^2\rb \lb \sum\limits_{i=1}^{n} \alpha_{i} \int_{\Omega} \lb \PD_{x_{i}} u\ \overline{u}+ u\ \PD_{x_{i}} \overline{u} \rb d x\rb\\
&=-2\lambda \lb \lambda^2+ \o^2\rb \lb \sum\limits_{i=1}^{n} \alpha_{i} \int_{\Omega} \PD_{ x_{i}}\lb |u|^{2} \rb d x\rb\\
&= -2\lambda\lb \lambda^2+ \o^2\rb  \int\limits_{\PD \O}\left( \alpha \cdot \nu \right) \ |u|^2 \ dS.
\end{align*}
Since $u=0$ on $\PD \O$, we have that $C+D=0$.

Now let us consider the first term: 
\begin{align*}
-2\lambda \cdot A& =-2\lambda\sum\limits_{i=1}^{n} \alpha_{i}\int_{\Omega} \PD_{x_{i}} u\ \Delta \overline{u} \ d x\\
&=2\lambda \sum\limits_{i=1}^{n} \alpha_{i}\int_{\Omega} u \ \PD_{ x_{i}} \lb \Delta \overline{u}\rb d x -2\lambda \int\limits_{\PD \O} \left( \alpha \cdot \nu\right) u \Delta \overline{u} \ dS.\\
\intertext{Note that the second term on the right above is $0$ and the Laplacian and the partial derivative on the first integral can be interchanged. We get}
-2\lambda \cdot A&=2\lambda \sum\limits_{i=1}^{n} \alpha_{i}\int_{\Omega} u\ \Delta \lb \PD_{x_{i}}\overline{u} \rb\ d x\\
&=2\lambda \sum\limits_{i=1}^{n} \alpha_{i} \int_{\Omega} \Delta u \ \PD_{x_{i}} \overline{u}\ d x +{2\lambda \sum\limits_{i=1}^{n} \alpha_{i} \int\limits_{\PD \Omega} u \ \PD_{\nu}\lb  \PD_{x_{i}} \overline{u}\rb \ d S} -2\lambda\sum\limits_{i=1}^{n} \alpha_{i} \int\limits_{\PD \Omega} \PD_{\nu} u \ \PD_{x_{i}} \overline{u}\ d S. 
\intertext{The second expression on the right is $0$ since $u=0$ on $\PD \O$.}
\end{align*}
Therefore we have 
\[
I+II=-2\lambda \int\limits_{\PD \O} \left(\PD_{\nu} u\right) \left( \alpha \cdot \n \overline{u}\right) \ dS.
\]
Now at each point $x\in \PD \O$, let us write 
\[
\alpha = \left( \alpha \cdot \nu(x)\right) \nu(x) + T(x),
\]
where $T(x)$ is a vector field tangent to $\PD \O$ at $x$. Since $u=0$ on $\PD \O$, we have that $ T(x)\cdot \nabla \overline{u}(x) =0$.
Hence we get 
\[
\lVert e^{\lambda \vp} \lb -\Delta - \o^2\rb  e^{-\lambda \vp} u\lVert_{L^{2}(\Omega)}^2= \lVert\lb  \Delta + \lambda^2+\o^2\rb u\rVert_{L^{2}(\Omega)}^2+4 \lambda^2 \lVert \alpha \cdot \nabla u\rVert_{L^{2}(\Omega)}^2-2\lambda \int\limits_{\PD \O} \left( \alpha \cdot \nu\right) \PD_{\nu}u\  \overline{\PD_{\nu}u}\  dS.
\]
Using Poincar\'e  inequality, we have
\begin{align*}
\lVert e^{\lambda \vp}(\Delta +\o^2) e^{-\lambda \vp}u\rVert^2_{L^{2}(\O)} &\geq  4\lambda^2\lVert \alpha \cdot \nabla u\rVert_{L^{2}(\Omega)}^2-2 \lambda  \int_{\partial \Omega} \left( \alpha \cdot \nu \right)  \partial_{\nu}  u \  \overline{\partial_{\nu}u} \ dS\\
&\geq C \lambda^2 \|u \|^{2}_{L^{2}(\Omega)} -2 \lambda  \int_{\partial \Omega} \left( \alpha \cdot \nu \right) \partial_{\nu}  u \  \overline{\partial_{\nu}u}\ dS,
\end{align*}
where the constant $C$ is independent of the frequency $\omega $.\\
We can rewrite the above inequality as
\begin{align*}
-2 \lambda  \int_{\partial \Omega_{-}} \left( \alpha \cdot \nu \right)  \partial_{\nu}  u \   \overline{\partial_{\nu}u} \ dS+C \lambda^2 \|u \|^{2}_{L^{2}(\Omega)} &\leq \lVert e^{\lambda \vp}(\Delta +\o^2) e^{-\lambda \vp}u\rVert^2_{L^{2}(\O)} \\
&+2 \lambda  \int_{\partial \Omega_{+}} \left( \alpha \cdot \nu \right)  \partial_{\nu}  u \  \overline{\partial_{\nu}u}\ dS.
\end{align*}	
Now 	
\begin{equation}
\lVert e^{\lambda \vp}(\Delta +\o^2) e^{-\lambda \vp}u\rVert_{L^{2}(\O)} \leq \lVert e^{\lambda \vp}(\Delta +\o^2-q) e^{-\lambda \vp}u\rVert_{L^{2}(\O)} +\|qu \|_{L^{2}(\Omega)},
\notag
\end{equation}
and hence
\begin{equation}
\begin{aligned}
\lVert e^{\lambda \vp}(\Delta +\o^2) e^{-\lambda \vp}u\rVert^2_{L^{2}(\O)} &\leq 2\lVert e^{\lambda \vp}(\Delta +\o^2-q) e^{-\lambda \vp}u\rVert^2_{L^{2}(\O)} +2\|qu \|^2_{L^{2}(\Omega)} \\
&\leq 2\lVert e^{\lambda \vp}(\Delta +\o^2-q) e^{-\lambda \vp}u\rVert^2_{L^{2}(\O)} +2\| q\|^{2}_{L^{\infty}(\Omega)}\|u \|^2_{L^{2}(\Omega)}.
\end{aligned}
\notag
\end{equation}
Choosing $\lambda $ large enough, we derive
\begin{align*}
&-2 \lambda  \int_{\partial \Omega_{-}} \left( \alpha \cdot \nu \right)  \partial_{\nu}  u \  \overline{\partial_{\nu}u}\ dS+C \lambda^2 \|u \|^{2}_{L^{2}(\Omega)} \leq 2\lVert e^{\lambda \vp}(\Delta +\o^2-q) e^{-\lambda \vp}u\rVert^2_{L^{2}(\O)} +2 \lambda  \int_{\partial \Omega_{+}} \left( \alpha \cdot \nu \right) \partial_{\nu}  u \  \overline{\partial_{\nu}u}\ dS.
\end{align*}	
Now the estimate	
\begin{align*}
&-\frac{1}{ \lambda } \int_{\partial \Omega_{-}} \left( \alpha \cdot \nu \right) \partial_{\nu}  u \  \overline{\partial_{\nu}u}\ dS\ +C  \|u \|^{2}_{L^{2}(\Omega)} \leq \frac{1}{\lambda^2}\lVert e^{\lambda \vp}(\Delta +\o^2-q) e^{-\lambda \vp}u\rVert^2_{L^{2}(\O)} +\frac{1}{\lambda} \int_{\partial \Omega_{+}} \left( \alpha \cdot \nu \right)  \partial_{\nu}  u \  \overline{\partial_{\nu}u} \ dS
\end{align*}	
follows.
\end{proof}

\begin{remark}
Note that we can use the linear Carleman weight $-\vp $  instead of $\vp $ in the previous inequality which would give us
\begin{align*}
&\frac{1}{ \lambda } \int_{\partial \Omega_{+}} \left(\alpha \cdot \nu \right) \partial_{\nu}  u \  \overline{\partial_{\nu}u}\ dS +C  \|u \|^{2}_{L^{2}(\Omega)} \\
&\qquad \qquad \qquad \leq \frac{1}{\lambda^2}\lVert e^{-\lambda \vp}(\Delta +\o^2-q) e^{\lambda \vp}u\rVert^2_{L^{2}(\O)} +\frac{1}{\lambda} \int_{\partial \Omega_{-}} \left(-\alpha \cdot \nu \right) \partial_{\nu}  u \  \overline{\partial_{\nu}u} \ dS.
\end{align*}	
Choosing $\tilde{u}=e^{\lambda \vp}u $, and using the fact that $\tilde{u}=0$ on $\PD\O$, we derive the estimate	
\begin{equation}
\begin{aligned}
&\frac{1}{ \lambda } \left\langle \sqrt{\alpha \cdot \nu } \ e^{-\lambda \vp}\partial_{\nu}  \tilde{u} , \sqrt{\alpha \cdot \nu }\ e^{-\lambda \vp} \partial_{\nu}\tilde{u} \right\rangle_{\partial \Omega_{+}} +C  \|e^{-\lambda \vp}\tilde{u} \|^{2}_{L^{2}(\Omega)} \\
&\qquad \leq \frac{1}{\lambda^2}\lVert e^{-\lambda \vp}(\Delta +\o^2-q) \tilde{u} \rVert^2_{L^{2}(\O)} +\frac{1}{\lambda} \left\langle \sqrt{-(\alpha \cdot \nu) }\ e^{-\lambda \vp}\partial_{\nu}  \tilde{u} ,\sqrt{-(\alpha \cdot \nu) }\ e^{-\lambda \vp} \partial_{\nu}\tilde{u} \right\rangle_{\partial \Omega_{-}}.
\end{aligned}
\label{boundary}
\end{equation}
\end{remark}
The following version of Green's identity can be derived following \cite{Al} and we skip the proof here.
\begin{lemma}
Let $u_{1},u_{2} $ satisfy \eqref{Elliptic PDE}  with $q_{1} $, $q_{2} $ respectively and $v$ satisfy $\mathcal{L}^{*}_{q_{1}}v=0 $.\\
Then 
\begin{equation}
\int_{\Omega} \left(q_{1}-q_{2} \right) u_{2}\overline{v}\ dx=\int_{\partial \Omega} \partial_{\nu}(u_{1}-u_{2}) \overline{v}\ dS.
\label{GR1}
\end{equation}
\end{lemma}
We shall also use the following result due to Sylvester and Uhlmann (see \cite{SU} and also \cite{ILW,INUW}) on the existence of CGO solutions for \eqref{Elliptic PDE}. 
\begin{theorem}\label{SU-est} Let $s>\frac{n}{2}$ be an integer and $\zeta \in \mathbb{C}^{n} $ satisfy $\zeta \cdot \zeta=\omega^2 $. Then there exist constants $C_1$ and $C_2$ (independent of $\o$ and only depending on $s$ and $\Omega $) such that if $|\zeta|> C_2\lVert q\rVert_{H^{s}(\O)}$, then there exists a solution to \eqref{Elliptic PDE} of the form
\[
u(x) =e^{ix\cdot \zeta}\lb 1+ r(x,\zeta)\rb, 
\]
with $r$ satisfying the following estimate: 
\[
\lVert r\rVert_{H^{s}(\O)}\leq \frac{C_1}{|\zeta|}\lVert q\rVert_{H^{s}(\O)}.
\]
\end{theorem}
The idea is to choose $\zeta $ suitably and use the above result to infer the existence of CGO solutions $u$ with the remainder term $r$ satisfying the above estimates.\\
\newline
Since we are dealing with the partial data case, suitable analytic continuation results need to be used to derive the stability estimates. We shall use the following analytic continuation result due to Vessella (see \cite{Ve} and also \cite{Heck-Wang-StabilityPaper}). 
\begin{theorem}\label{ves}
Let $D \subset \mathbb{R}^{n} $ be a bounded open connected set such that for a positive number $r_{0} $ the set $D_{r}=\{x \in D: \ d(x,\partial D)>r \} $ is connected for every $r \in [0,r_{0}] $. Let $E \subset D $ be an open set such that $d(E,\partial D)\geq d_{0}>0 $. Let $f$ be an analytic function on $D $ with the property that 
\[\vert D^{\alpha} f(x) \vert \leq \frac{C \alpha !}{\mu^{|\alpha|}} \ \text{for}\ x \in D , \alpha \in (\mathbb{N}\cup \{0\})^{n},\]
where $\mu, C $ are positive numbers. Then \[\vert f(x) \vert \leq (2C)^{1-\gamma_{1}\left(\frac{\vert E \vert}{\vert D \vert} \right)} \left(\sup_{E} \vert f(x)\vert \right)^{\gamma_{1}\left(\frac{\vert E \vert}{\vert D \vert} \right)}, \]
where $\vert E \vert $ and $\vert D \vert $ denote the Lebesgue measure of $E $ and $D $, respectively, $\gamma_{1} \in (0,1)$ and $\gamma_{1}$ depends only on $d_{0}, \text{diam}(D), n, r_{0},\mu $ and $d(x,\partial D) $.
\end{theorem}
\section{Stability estimates}
In this section, we prove Theorem \ref{Main-result}. We introduce suitable CGO solutions as follows. 

Let
\begin{align}
\label{z1}\zeta_{1}&=\frac{1}{2} \xi +i \lambda\alpha- \left(\omega^2+\lambda^2-\frac{\vert \xi \vert^2}{4} \right)^{\frac{1}{2}} \beta, \\
\label{z2}\zeta_{2}&=-\frac{1}{2} \xi -i \lambda \alpha- \left(\omega^2+\lambda^{2}-\frac{\vert \xi \vert^2}{4} \right)^{\frac{1}{2}} \beta ,
\end{align}
for $\omega^2+\lambda^2 > \frac{\vert \xi \vert^2}{4} $. 
Then $\zeta_j \cdot \zeta_j=\omega^2$. Using theorem \ref{SU-est}, provided $\vert  \zeta_{j} \vert > C_{2} \| q_{j} \|_{H^{s}(\Omega)} $, we have solutions $ v$ and $u_{2} $ to $\mathcal{L}^{*}_{q_1} v=0  $ and $\mathcal{L}_{q_2} u_{2}=0 $ of the form
\begin{equation}
\begin{aligned}
v(x)&=e^{i \zeta_{1}\cdot x}(1+r_{1}(x,\zeta_{1};\lambda)), \quad 
u_{2}(x)= e^{i \zeta_{2}\cdot x}(1+r_{2}(x,\zeta_{2};\lambda)),
\end{aligned}
\notag
\end{equation}
where the remainder terms $r_{j},\ j=1,2 $ satisfy the estimates
\begin{equation}
\|r_{j} \|_{H^{s}(\Omega)} \leq \frac{C_{1}}{\vert  \zeta_{j} \vert} \| q_{j} \|_{H^{s}(\Omega)}.
\label{SU}
\end{equation}

Note that $\vert \zeta_{j} \vert= \left(\omega^{2}+2\lambda^{2} \right)^{\frac{1}{2}} $.
Therefore provided $\left(\omega^{2}+2\lambda^{2} \right)^{\frac{1}{2}}> C_{2}M $, we have the estimate 
\begin{equation}
\|r_{j} \|_{H^{s}(\Omega)} \leq \frac{C_{1}}{\vert  \zeta_{j} \vert} \| q_{j} \|_{H^{s}(\Omega)} \leq \frac{C_{1} M}{C_{2} M} \leq C.
\label{SU-1}
\end{equation}
We rewrite \eqref{GR1} as
\begin{equation}
\begin{aligned}
&\int_{\Omega} \left(q_{1}-q_{2} \right) u_{2}\overline{v}\ dx=\int_{\partial \Omega_{-,\epsilon}} \partial_{\nu}(u_{1}-u_{2}) \overline{v}\ dS+ \int_{\partial \Omega_{+,\epsilon}} \partial_{\nu}(u_{1}-u_{2}) \overline{v}\ dS.
\end{aligned}
\label{GR2}
\end{equation}
We can estimate the $H^{1}$ and $H^{2}$ norms of $v$ and $u_{2}$ in the following manner.

Let $R\geq 1$ be such that $\Omega \subset B(0,R) $. {Then since $\vert e^{i\zeta_{j} \cdot x}\vert \leq e^{\vert \mathrm{Im} \ \zeta_{j} \vert \vert x \vert} \leq e^{\lambda R} $, we have
\begin{equation}
\begin{aligned}
\|v \|_{H^{1}(\Omega)} &\leq \|e^{i \zeta_{1} \cdot x}(1+r_{1}) \|_{L^{2}(\Omega)}+ \sum_{k=1}^{n} \|e^{i \zeta_{1} \cdot x} \partial_{x_{k}} r_{1}+ i \zeta_{1k} e^{i \zeta_{1}\cdot x}(1+r_{1})\|_{L^{2}(\Omega)}\\
&\leq C \vert \zeta_{1}\vert \vert e^{i \zeta_{1} \cdot x}\vert \|1+r_{1} \|_{H^{1}(\Omega)}\\
&\leq C \left( \omega^2+2\lambda^2\right)^{\frac{1}{2}} e^{\lambda R } \|1+r_{1} \|_{H^{s}(\Omega)} \\
&\leq C \left( \omega^2+2\lambda^2\right)^{\frac{1}{2}} e^{\lambda R } \quad \text{(using \eqref{SU-1})}
\end{aligned}
\end{equation}
and similarly
\begin{equation}
\begin{aligned}
\|v \|_{H^{2}(\Omega)}&\leq C \left( \omega^2+2\lambda^2\right) e^{\lambda R}.
\end{aligned}    
\label{SU-2}
\end{equation}
Using these, we estimate the terms in the right-hand side of \eqref{GR2} as follows. For the integral over $\partial \Omega_{-,\epsilon} $ we note that
\begin{equation}
\begin{aligned}
\left\vert \int_{\partial \Omega_{-,\epsilon}} \partial_{\nu}(u_{1}-u_{2}) \overline{v}\ dS \right\vert &\leq \|\partial_{\nu}(u_{1}-u_{2}) \|_{L^{2}(\partial \Omega_{-,\epsilon})} \|v \|_{L^{2}(\partial \Omega_{-,\epsilon})} \\
&\leq C\|\partial_{\nu}(u_{1}-u_{2}) \|_{H^{\frac{1}{2}}(\partial \Omega_{-,\epsilon})} \|v \|_{H^{\frac{1}{2}}(\partial \Omega_{-,\epsilon})} \\
&\leq C \|\left(\wt{\Lambda}_{q_{1}}-\wt{\Lambda}_{q_{2}}\right) (f) \|_{H^{\frac{1}{2}}(\partial \Omega)} \|v \|_{H^{1}(\Omega)}\\
&\leq C \|\wt{\Lambda}_{q_{1}}-\wt{\Lambda}_{q_{2}} \| \| f\|_{H^{\frac{3}{2}}(\partial \Omega)} \| v\|_{H^{1}(\Omega)} \\
&\leq C \|\wt{\Lambda}_{q_{1}}-\wt{\Lambda}_{q_{2}} \| \| u_{2}\|_{H^{2}( \Omega)} \| v\|_{H^{1}(\Omega)}\\
&\leq C \left( \omega^2+2\lambda^2\right)^{\frac{3}{2}} e^{2\lambda R }\ \|\wt{\Lambda}_{q_{1}}-\wt{\Lambda}_{q_{2}} \| .
\end{aligned}
\label{est-99}
\end{equation}

To estimate the integral over $\partial \Omega_{+,\epsilon} $, we shall use the boundary Carleman estimate \eqref{boundary}. First of all, we note that 
\begin{equation}
\begin{aligned}
\left\vert \int_{\partial \Omega_{+,\epsilon}} \partial_{\nu}(u_{1}-u_{2}) \overline{v}\ dS \right\vert &=\left\vert  \int_{\partial \Omega_{+,\epsilon}} e^{-\lambda \alpha \cdot x}\  \partial_{\nu}(u_{1}-u_{2}) \ e^{\lambda \alpha \cdot x} \ \overline{v}\ dS \right\vert \\
&\leq \| e^{-\lambda \alpha \cdot x}\  \partial_{\nu}(u_{1}-u_{2})  \|_{L^{2}(\partial \Omega_{+,\epsilon})} \|e^{\lambda \alpha \cdot x} \ \overline{v} \|_{L^{2}(\partial \Omega_{+,\epsilon})}.
\end{aligned}
\label{est-100}
\end{equation}
Now
\begin{equation}
\begin{aligned}
&\zeta_{1}=\frac{1}{2} \xi +i \lambda\alpha- \left(\omega^2+\lambda^2-\frac{\vert \xi \vert^2}{4} \right)^{\frac{1}{2}} \beta \\
&\overline{\zeta}_{1}=\frac{1}{2} \xi -i \lambda\alpha- \left(\omega^2+\lambda^2-\frac{\vert \xi \vert^2}{4} \right)^{\frac{1}{2}} \beta \\
&\overline{i} \overline{\zeta}_{1}=-i \left[ \frac{1}{2} \xi -i \lambda\alpha- \left(\omega^2+\lambda^2-\frac{\vert \xi \vert^2}{4} \right)^{\frac{1}{2}} \beta \right]=-\lambda \alpha +i \left[-\frac{1}{2} \xi +\left(\omega^2+\lambda^2-\frac{\vert \xi \vert^2}{4} \right)^{\frac{1}{2}}\beta \right].
\end{aligned}
\notag
\end{equation}
Therefore
\begin{equation}
\begin{aligned}
e^{\lambda \alpha \cdot x} \overline{v}=e^{\lambda \alpha \cdot x} \overline{e^{i \zeta_{1} \cdot x}} \left( 1+\overline{r}_{1}(x,\zeta_{1}; \lambda)\right)= e^{i \left[-\frac{1}{2} \xi +\left(\omega^2+\lambda^2-\frac{\vert \xi \vert^2}{4} \right)^{\frac{1}{2}} \beta \right]\cdot x} \left( 1+\overline{r}_{1}(x,\zeta_{1}; \lambda)\right),
\end{aligned}
\notag
\end{equation}
and
\[\| e^{\lambda \alpha \cdot x} \overline{v}\|_{L^{2}(\partial \Omega_{+,\epsilon})}=\|1+\overline{r}_{1}(x,\zeta_{1};\lambda) \|_{L^{2}(\partial \Omega_{+,\epsilon})}. \]
Using \eqref{SU} and trace theorem, 
\begin{equation}
\begin{aligned}
&\| \overline{r_1}\|_{L^{2}(\partial \Omega_{+,\epsilon})}\leq \| \overline{r_1}\|_{L^{2}(\partial \Omega_{+})} \leq C \| \overline{r_1}\|_{H^{1}(\Omega)} \leq \frac{C}{\vert \zeta_{1}\vert} \cdot M \leq C,
\end{aligned}
\notag
\end{equation}
where we use the fact that $\vert \zeta_{1} \vert=\left(\omega^{2}+2\lambda^{2} \right)^{\frac{1}{2}} > 1 $.\\
Using this in \eqref{est-100}, we have
\begin{equation}
\left\vert \int_{\partial \Omega_{+,\epsilon}} \partial_{\nu}(u_{1}-u_{2}) \overline{v}\ dS \right\vert \leq C \| e^{-\lambda \alpha \cdot x}\  \partial_{\nu}(u_{1}-u_{2})  \|_{L^{2}(\partial \Omega_{+,\epsilon})}.
\label{est-101}
\end{equation}
From the boundary Carleman estimate \eqref{boundary}, we have 
\begin{align*}
\frac{1}{\lambda} \|\sqrt{\alpha \cdot \nu}\ e^{-\lambda \vp} \ \partial_{\nu} \tilde{u} \|^{2}_{L^{2}(\partial \Omega_{+})}+C \|e^{-\lambda \varphi} \tilde{u} \|^{2}_{L^{2}(\Omega)} & \leq \frac{1}{\lambda^2} \|e^{-\lambda \vp} \left(\Delta+\omega^{2}-q \right)\tilde{u} \|^{2}_{L^{2}(\Omega)}\\
&\quad +\frac{1}{\lambda} \|\sqrt{-(\alpha \cdot \nu)}\ e^{-\lambda \vp} \ \partial_{\nu} \tilde{u} \|^{2}_{L^{2}(\partial \Omega_{-})} .
\end{align*}
This gives,
\begin{align}
\notag & \|\sqrt{\alpha \cdot \nu}\ e^{-\lambda \vp} \ \partial_{\nu} \tilde{u} \|^{2}_{L^{2}(\partial \Omega_{+})} \leq \frac{1}{\lambda} \|e^{-\lambda \vp} \left(\Delta+\omega^{2}-q \right)\tilde{u} \|^{2}_{L^{2}(\Omega)}+ \|\sqrt{-(\alpha \cdot \nu)}\ e^{-\lambda \vp} \ \partial_{\nu} \tilde{u} \|^{2}_{L^{2}(\partial \Omega_{-})} \\
\intertext{from which, we have}
& \|\sqrt{\alpha \cdot \nu}\ e^{-\lambda \vp} \ \partial_{\nu} \tilde{u} \|_{L^{2}(\partial \Omega_{+})} \leq \frac{1}{\sqrt{\lambda}} \|e^{-\lambda \vp} \left(\Delta+\omega^{2}-q \right)\tilde{u} \|_{L^{2}(\Omega)}+ \|\sqrt{-(\alpha \cdot \nu)}\ e^{-\lambda \vp} \ \partial_{\nu} \tilde{u} \|_{L^{2}(\partial \Omega_{-})}.
\label{est-102}
\end{align}
Now on $\partial \Omega_{+,\epsilon} $, we have $\alpha \cdot \nu > \epsilon $ and hence 
\begin{align}
\notag \sqrt{\epsilon} \|e^{-\lambda \vp} \partial_{\nu} \tilde{u} \|_{L^{2}(\partial \Omega_{+,\epsilon})}&\leq \|\sqrt{\alpha \cdot \nu}\ e^{-\lambda \vp} \ \partial_{\nu} \tilde{u} \|_{L^{2}(\partial \Omega_{+})} \\
\notag & \leq \frac{1}{\sqrt{\lambda}} \|e^{-\lambda \vp} \left(\Delta+\omega^{2}-q \right)\tilde{u} \|_{L^{2}(\Omega)}+ \|\sqrt{-(\alpha \cdot \nu)}\ e^{-\lambda \vp} \ \partial_{\nu} \tilde{u} \|_{L^{2}(\partial \Omega_{-})}. \\
\intertext{This gives}
\|e^{-\lambda \vp} \partial_{\nu} \tilde{u} \|_{L^{2}(\partial \Omega_{+,\epsilon})} &\leq \frac{1}{\sqrt{\epsilon}} \left(\frac{1}{\sqrt{\lambda}} \|e^{-\lambda \vp} \left(\Delta+\omega^{2}-q \right)\tilde{u} \|_{L^{2}(\Omega)}+ \sqrt{-\inf_{\partial \Omega_{-}}(\alpha \cdot \nu)}\|\ e^{-\lambda \vp} \ \partial_{\nu} \tilde{u} \|_{L^{2}(\partial \Omega_{-,\epsilon})} \right).
\label{est-103}
\end{align}
Choosing $\tilde{u}=u_{1}-u_{2} $ and $q=q_{1} $ in \eqref{est-103}, from \eqref{est-101} we can infer
\begin{equation}
\begin{aligned}
\left\vert \int_{\partial \Omega_{+,\epsilon}} \partial_{\nu}(u_{1}-u_{2}) \overline{v}\ dS \right\vert &\leq  \frac{C}{\sqrt{\epsilon \lambda}} \|e^{-\lambda \vp} \left(\Delta+\omega^{2}-q_{1} \right) (u_{1}-u_{2}) \|_{L^{2}(\Omega)}\\
&\quad +\frac{C}{\sqrt{\epsilon}} \sqrt{-\inf_{\partial \Omega_{-}}(\alpha \cdot \nu)}\|\ e^{-\lambda \vp} \ \partial_{\nu} (u_{1}-u_{2}) \|_{L^{2}(\partial \Omega_{-,\epsilon})} .
\end{aligned}
\label{est-104}
\end{equation}
Using the facts $\left(\Delta+\omega^{2}-q_{1} \right)(u_{1}-u_{2})=(q_{1}-q_{2})u_{2}, $ and \[e^{-\lambda \alpha \cdot x}u_{2}=e^{i \left[-\frac{1}{2} \xi-\left(\omega^{2}+\lambda^{2}-\frac{\vert \xi \vert^{2}}{4} \right)^{\frac{1}{2}} \beta \right]} \left(1+r_{2}(x,\zeta_{2};\lambda) \right), \] we observe that
\begin{equation}
\|e^{-\lambda \vp} \left(\Delta+\omega^{2}-q_{1} \right) (u_{1}-u_{2}) \|_{L^{2}(\Omega)}= \|e^{-\lambda \vp} (q_{1}-q_{2}) u_{2} \|_{L^{2}(\Omega)}\leq \|(q_{1}-q_{2})\left(1+r_{2}(x,\zeta_{2};\lambda) \right) \|_{L^{2}(\Omega)}\leq C,
\notag
\end{equation}
and using this in \eqref{est-104}, we have
\begin{equation}
\begin{aligned}
\left\vert \int_{\partial \Omega_{+,\epsilon}} \partial_{\nu}(u_{1}-u_{2}) \overline{v}\ dS \right\vert &\leq \frac{C}{\sqrt{\epsilon \lambda}}+ \frac{C}{\sqrt{\epsilon}} \sqrt{-\inf_{\partial \Omega_{-}}(\alpha \cdot \nu)}\|\ e^{-\lambda \vp} \ \partial_{\nu} (u_{1}-u_{2}) \|_{L^{2}(\partial \Omega_{-,\epsilon})} \\
\Rightarrow \left\vert \int_{\partial \Omega_{+,\epsilon}} \partial_{\nu}(u_{1}-u_{2}) \overline{v}\ dS \right\vert &\leq \frac{C}{\sqrt{ \lambda}}+ C e^{R\lambda} \|\partial_{\nu} (u_{1}-u_{2}) \|_{L^{2}(\partial \Omega_{-,\epsilon})} \\
&\leq \frac{C}{\sqrt{ \lambda}}+ Ce^{2\lambda R} \left( \omega^2+2\lambda^2\right) \|\wt{\Lambda}_{q_{1}}-\wt{\Lambda}_{q_{2}} \|, 
\end{aligned}
\label{est-105}
\end{equation}
where the constant now depends on $\epsilon $. 

We extend $q_1$ and $q_2$ to be $0$ outside $\O$. Using these inequalities in \eqref{GR2}, we obtain the estimate 
\begin{align}\label{est-106}
\notag \left\vert \int_{\Rb^n} (q_{1}-q_{2})\ e^{-i\xi \cdot x}\ dx\right\vert & \leq \left\vert \int_{\partial \Omega_{+,\epsilon}} \partial_{\nu}(u_{1}-u_{2}) \overline{v}\ dS \right\vert + \left\vert \int_{\partial \Omega_{-,\epsilon}} \partial_{\nu}(u_{1}-u_{2}) \overline{v}\ dS\right\vert \\
\notag &\qquad + \left\vert \int_{\Omega} (q_{1}-q_{2}) e^{-i \xi \cdot x} \left(\overline{r_{1}}+r_{2}+\overline{r_{1}}r_{2} \right)\ dx \right\vert \\
\notag & \leq \frac{C}{\sqrt{ \lambda}}+ C e^{2\lambda R}\left( \omega^2+2\lambda^2\right)  
\lVert \wt{\Lambda}_{q_{1}}-\wt{\Lambda}_{q_{2}} \rVert +C \left( \omega^2+2\lambda^2\right)^{\frac{3}{2}} e^{2 \lambda R } \|\wt{\Lambda}_{q_{1}}-\wt{\Lambda}_{q_{2}} \| \\
\notag &\qquad +\frac{C}{\left(\omega^{2}+2\lambda^{2} \right)^{\frac{1}{2}}}\\
\notag &\leq C \left(\frac{1}{\sqrt{ \lambda}}+  \left( \omega^2+2\lambda^2\right)^{\frac{3}{2}} e^{2 \lambda R }\ \|\wt{\Lambda}_{q_{1}}-\wt{\Lambda}_{q_{2}} \| +\frac{1}{\left(\omega^{2}+2\lambda^{2} \right)^{\frac{1}{2}}} \right)\\
&\leq C \left( \left( \omega^2+2\lambda^2\right)^{\frac{3}{2}} e^{2 \lambda R } \|\wt{\Lambda}_{q_{1}}-\wt{\Lambda}_{q_{2}} \|+ \frac{1}{\sqrt{\lambda}}\right),
\end{align}
where we use the fact that $ \lambda\leq \left(\omega^{2}+2\lambda^{2}\right)^{1/2} $ for $\lambda \geq 1 $.
We perturb $\alpha$ in a small enough conic neighborhood.  Correspondingly, the vector $\xi$ chosen perpendicular to $\alpha$ would vary in a small conic neighborhood denoted by $V$. Now for all $\xi\in V$,  with the condition that $|\xi|\leq 2\sqrt{\o^2+\lambda^2}$, estimate \eqref{est-106} holds.  We will consider those $\xi$ such that $|\xi|<\lambda$. This would obviously imply that $|\xi|\leq 2\sqrt{\o^2+\lambda^2}$. Hence for all $\xi$ such that $|\xi|\leq \lambda$, \eqref{est-106} holds.



Denoting $q=q_1-q_2$, let us split 
$ \| q \|^{2}_{H^{-1}(\mathbb{R}^{n})}$ as 
\[
\| q \|^{2}_{H^{-1}(\mathbb{R}^{n})}= \left(\int_{\vert \xi \vert<\rho} \frac{\vert \wh{q}(\xi) \vert^2}{1+\vert \xi \vert^2}\ d\xi+\int_{\vert \xi \vert \geq \rho} \frac{\vert \wh{q}(\xi) \vert^2}{1+\vert \xi \vert^2}\ d\xi\right)
\]
with $\rho$ to be chosen later.

The second integral can be estimated as 
\[
\int_{\vert \xi \vert \geq \rho} \frac{\vert \wh{q}(\xi) \vert^2}{1+\vert \xi \vert^2}\ d\xi\leq \frac{1}{1+\rho^2}\int |\wh{q}(\xi)|^2 d \xi\leq \frac{1}{\rho^2}\int |\wh{q}(\xi)|^2 d \xi = \frac{1}{\rho^2}\lVert q\rVert_{L^{2}(\O)}^2\leq \frac{C}{\rho^2}.
\]
The first integral, we estimate using the result by Vessella (see Theorem \ref{ves} above) following the arguments in \cite{Heck-Wang-StabilityPaper}.
We have the following estimate for $\wh{q}(\xi)$ from \cite{Heck-Wang-StabilityPaper} for $\xi \in B(0,\rho)$: 
\[
\lvert \wh{q}(\xi)\rvert \leq C e^{n\rho(1-\theta)}\lVert \wh{q} \rVert_{L^{\infty}(V)}^{\theta},
\]
where $\theta\in (0,1)$ is a positive constant independent of $\o$. Using the above estimate, we have 
\begin{align*}
\int\limits_{\lvert \xi\rvert<\rho} \frac{\lvert \wh{q}(\xi)\rvert^{2}}{1+\lvert \xi\rvert^2} d \xi \leq \|\wh{q} \|^{2}_{L^{\infty}(B(0,\rho))} \int \limits_{|\xi|<\rho} \frac{1}{1+|\xi|^2} d \xi \leq C  \rho^{n} e^{2n\rho (1-\theta)}\| \wh{q}\|^{2\theta}_{L^{\infty}(V\cap B(0,\rho))}.
\end{align*}
Now
\begin{align}
\notag \| q \|^{\frac{2}{\theta}}_{H^{-1}(\mathbb{R}^{n})} & 
\notag \leq C \left(\rho^{\frac{n}{\theta}} e^{2n\rho \frac{1-\theta}{\theta}} \| \wh{q} \|^{2}_{L^{\infty}(V \cap B(0,\rho))}+\frac{1}{\rho^{\frac{2}{\theta}}} \right). \\
\intertext{Using the estimate for the Fourier transform of $q$ from \eqref{est-106}, we get,}
\notag  &\leq  C \left[\rho^{\frac{n}{\theta}} e^{2n\rho \frac{1-\theta}{\theta}} \left( \omega^2+2\lambda^2\right)^{3} e^{4 \lambda R } \|\wt{\Lambda}_{q_{1}}-\wt{\Lambda}_{q_{2}} \|^{2}+\rho^{\frac{n}{\theta}} e^{2n\rho \frac{1-\theta}{\theta}} \cdot \frac{1}{\lambda}+\frac{1}{\rho^{\frac{2}{\theta}}} \right]. \\
\intertext{Using the inequality $(\o^2 + 2\lambda^2)^{3}\leq C(\o^6+ \lambda^6)$, we get,}
\notag &\leq C \left[\rho^{\frac{n}{\theta}} e^{2n\rho \frac{1-\theta}{\theta}} \omega^6  e^{4 \lambda R } \|\wt{\Lambda}_{q_{1}}-\wt{\Lambda}_{q_{2}} \|^{2}+\rho^{\frac{n}{\theta}} e^{2n\rho \frac{1-\theta}{\theta}} \lambda^6  e^{4 \lambda R } \|\wt{\Lambda}_{q_{1}}-\wt{\Lambda}_{q_{2}} \|^{2}+\rho^{\frac{n}{\theta}} e^{2n\rho \frac{1-\theta}{\theta}} \cdot \frac{1}{\lambda}+\frac{1}{\rho^{\frac{2}{\theta}}} \right].\\
\intertext{Using $\lambda^6\leq e^{\lambda R}$, we have,}
\notag &\leq C \left[\rho^{\frac{n}{\theta}} e^{2n\rho \frac{1-\theta}{\theta}} \omega^6  e^{4 \lambda R } \|\wt{\Lambda}_{q_{1}}-\wt{\Lambda}_{q_{2}} \|^{2}+\rho^{\frac{n}{\theta}} e^{2n\rho \frac{1-\theta}{\theta}}   e^{5 \lambda R } \|\wt{\Lambda}_{q_{1}}-\wt{\Lambda}_{q_{2}} \|^{2}+\rho^{\frac{n}{\theta}} e^{2n\rho \frac{1-\theta}{\theta}} \cdot \frac{1}{\lambda}+\frac{1}{\rho^{\frac{2}{\theta}}} \right].\\
\intertext{Combining the first two expressions, using the fact that $\o>1$, we get,}
 &\leq  C \left[\underbrace{\omega^6 \rho^{\frac{n}{\theta}} e^{2n\rho \frac{1-\theta}{\theta}}   e^{5 \lambda R } \|\wt{\Lambda}_{q_{1}}-\wt{\Lambda}_{q_{2}} \|^{2}}_{\mathrm{I}}+\underbrace{\rho^{\frac{n}{\theta}} e^{2n\rho \frac{1-\theta}{\theta}} \cdot \frac{1}{\lambda}}_{\mathrm{II}}+\underbrace{\frac{1}{\rho^{\frac{2}{\theta}}}}_{\mathrm{III}} \right].
\label{est-107}
\end{align}

Let us choose $\lambda$ such that the terms II  and III in \eqref{est-107} are equal. Then
\Beq\label{Lambda}
\lambda=\rho^{\frac{n+2}{\theta}} e^{2n\rho\frac{1-\theta}{\theta}}.
\Eeq
Note that $\lambda$ depends on $\rho$. We need to choose $\rho$ suitably since the choice of $\lambda$ must satisfy certain conditions to apply the Carleman estimate and the CGO solutions guaranteed by theorem \ref{SU-est}.

With this in mind, let us estimate the first term in \eqref{est-107} with the choice of $\lambda$ from \eqref{Lambda}
above. We have 
\begin{align}
\mathrm{I}
&\leq \omega^6 e^{\frac{n}{\theta}\rho} e^{2n\rho \frac{1-\theta}{\theta}}   e^{5R \left[\rho^{\frac{n+2}{\theta}} e^{2n\rho\frac{1-\theta}{\theta}} \right] } \|\wt{\Lambda}_{q_{1}}-\wt{\Lambda}_{q_{2}} \|^{2}, 
\mbox{ where we use the fact that $\rho^{\frac{n}{\theta}}\leq e^{\frac{n}{\theta} \rho}$.} \\
\intertext{Again using, $\frac{n}{\theta}\rho+2n\frac{1-\theta}{\theta}\rho\leq \exp\lb \frac{n}{\theta}\rho+2n\frac{1-\theta}{\theta}\rho\rb$, $5R\leq e^{5R}$ and $\rho^{\frac{n+2}{\theta}}\leq \exp\lb \frac{n+2}{\theta}\rho\rb$, we get,}
&\leq \omega^6 \left[\exp\left\{ \exp\lb \frac{n}{\theta}\rho+2n\frac{1-\theta}{\theta}\rho\rb+\exp{\lb 5R+\frac{n+2}{\theta}\rho+2n\rho \frac{1-\theta}{\theta}\rb}\right\} \|\wt{\Lambda}_{q_{1}}-\wt{\Lambda}_{q_{2}} \|^{2}\right]. \\
\intertext{Since $\rho$ is chosen greater than $1$,}
&\leq \omega^6 \left[\exp\left\{ \exp\lb \frac{n}{\theta}\rho+2n\frac{1-\theta}{\theta}\rho\rb+\exp{\lb 5R\rho+\frac{n+2}{\theta}\rho+2n\rho \frac{1-\theta}{\theta}\rb}\right\} \|\wt{\Lambda}_{q_{1}}-\wt{\Lambda}_{q_{2}} \|^{2}\right]. \\
\intertext{Using $e^A+e^B\leq 1 +e^{A+B}$, we have,}
&\leq C \omega^6 \exp\lb \exp\lb \left(\frac{n}{\theta}+4n\frac{1-\theta}{\theta}+5R+\frac{n+2}{\theta} \right)\rho\rb \rb \|\wt{\Lambda}_{q_{1}}-\wt{\Lambda}_{q_{2}} \|^{2}.\\ 
\intertext{Denoting $K=\left(\frac{n}{\theta}+4n\frac{1-\theta}{\theta}+5R+\frac{n+2}{\theta} \right)$, we rewrite}
&=\ C \omega^6 \exp\lb \exp(K \rho)\rb \|\wt{\Lambda}_{q_{1}}-\wt{\Lambda}_{q_{2}} \|^{2}.
\end{align}
We make the following choice for $\rho$: 
\Beq\label{rho2}
\rho= \frac{1}{K}\ln \left(\ln \ \omega+\vert \ln\ \|\wt{\Lambda}_{q_{1}}-\wt{\Lambda}_{q_{2}} \| \vert \right).
\Eeq
We assume that $\lVert \wt{\Lambda}_{q_{1}}-\wt{\Lambda}_{q_{2}} \rVert$ satisfies the following: 
\[
\lVert \wt{\Lambda}_{q_{1}}-\wt{\Lambda}_{q_{2}} \rVert \leq \exp\lb -\exp\lb K \wt{\lambda}^{1/L}\rb\rb,
\]
where $\wt{\lambda}> \mbox{ max} \lb 1, \lambda_0, C_2 M\rb$ and $L=\frac{3n-2n\theta+2}{\theta}$.
Note that, with this choice, $\lVert \wt{\Lambda}_{q_{1}}-\wt{\Lambda}_{q_{2}} \rVert<1$. 
Then 
\[
\rho = \frac{1}{K}\lb \ln \lb \ln \ \omega+\vert \ln\ \|\wt{\Lambda}_{q_{1}}-\wt{\Lambda}_{q_{2}} \| \vert \rb \rb\geq \frac{1}{K}\lb \ln \lb \vert \ln\ \|\wt{\Lambda}_{q_{1}}-\wt{\Lambda}_{q_{2}} \| \vert \rb \rb \geq \wt{\lambda}^{1/L}.
\]
Now 
\[
\lambda =\rho^{\frac{n+2}{\theta}} e^{2n\rho\frac{1-\theta}{\theta}}\geq  \rho^{\frac{n+2}{\theta}} \rho^{2n\frac{1-\theta}{\theta}} = \rho^{\frac{3n-2n\theta+2}{\theta}} = \rho^{L}, \mbox{ where $L$ was defined above.}
\] 
Then by the above $\lambda\geq \rho^{L}\geq \wt{\lambda} > \lambda_{0}$.  This choice of $\lambda$ is required in the Carleman estimate. 
Then with the inequalities, $\lambda \geq \rho^{L}\geq \wt{\lambda}\geq \lambda_0$, we have $\lambda \geq \lambda_0$.

Also note that if $\xi$ is chosen such that $|\xi|\leq \rho$, since $L$ defined above satisfies $L\geq 1$ and $\lambda \geq 1$, we have 
\[
|\xi|\leq \rho \leq \lambda ^{1/L}\leq \lambda \leq 2\lb \lambda^2 +\o^2\rb.
\]
Hence the vectors \eqref{z1} and \eqref{z2} are well defined, as well as the estimate in \eqref{est-106} can be applied.

Finally,  in Theorem \ref{SU-est}, we require $|\zeta|\geq C_2 \lVert q\rVert_{H^{s}(\O)}$.  Recall that $|\zeta|=\sqrt{\o^2+2\lambda^2}$.  Since we have taken $\wt{\lambda} \geq C_2M$, where $M$ is the bound on the potentials, our choice of $\lambda$ satisfies this inequality as well.

Now going back to the proof of the theorem, we have 
\begin{equation}
\omega^6 \rho^{\frac{n}{\theta}} e^{2n\rho \frac{1-\theta}{\theta}}   e^{5 \lambda R } \|\wt{\Lambda}_{q_{1}}-\wt{\Lambda}_{q_{2}} \|^{2} \leq C \omega^7 \|\wt{\Lambda}_{q_{1}}-\wt{\Lambda}_{q_{2}} \|
\notag
\end{equation}
and using this in \eqref{est-107} we obtain
\begin{equation}
\begin{aligned}
\| q \|^{\frac{2}{\theta}}_{H^{-1}(\O)}\leq \| q \|^{\frac{2}{\theta}}_{H^{-1}(\mathbb{R}^{n})}& \leq  C\left[\omega^7 \|\wt{\Lambda}_{q_{1}}-\wt{\Lambda}_{q_{2}} \|+\frac{1}{\left[\frac{1}{K} \ln \left(\ln \ \omega+\vert \ln\ \|\wt{\Lambda}_{q_{1}}-\wt{\Lambda}_{q_{2}} \| \vert \right) \right]^{\frac{2}{\theta}}} \right],
\end{aligned}
\label{est-113}
\end{equation}
whenever $ \|\tilde{\Lambda}_{q_{1}}-\tilde{\Lambda}_{q_{2}} \|<\delta := \exp\lb -\exp\lb K \wt{\lambda}^{1/L}\rb\rb $.
\newline
The estimate for the case $\|\tilde{\Lambda}_{q_{1}}-\tilde{\Lambda}_{q_{2}} \| \geq \delta $ can be easily deduced as follows. Recall that $\delta $ is independent of $\omega $. We use the continuous inclusions $L^{\infty}(\Omega) \hookrightarrow L^{2}(\Omega) \hookrightarrow H^{-1}(\Omega) $ to derive
\begin{equation}
\| q_{1}-q_{2} \|_{H^{-1}(\Omega)} \leq C \| q_{1}-q_{2} \|_{L^{\infty}(\Omega)} \leq \frac{2CM}{\delta^{\frac{\theta}{2}}} \delta^{\frac{\theta}{2}} \leq \frac{2CM}{\delta^{\frac{\theta}{2}}} \|\tilde{\Lambda}_{q_{1}}-\tilde{\Lambda}_{q_{2}} \|^{\frac{\theta}{2}}.
\label{est-114}   
\end{equation}
The $L^{\infty} $ norm of $q_{1}-q_{2} $ can now be estimated using interpolation. We recall that given $k_{0},k,k_{1}$ satisfying $k_{0}<k_{1} $ and $k=(1-p)k_{0}+pk_{1} $, where $p \in (0,1) $, the interpolation theorem gives the following estimate for the $H^{k} $ norm of a function $f$:
\[\| f\|_{H^{k}(\Omega)} \leq \|f \|_{H^{k_{0}}(\Omega)}^{1-p} \cdot \|f \|_{H^{k_{1}}(\Omega)}^{p} .\] 
To apply this in our case, let $\eta>0 $ be such that $s=\frac{n}{2}+2\eta $. We choose $k_{0}=-1, k_{1}=s $ and $k=\frac{n}{2}+\eta=s-\eta $. Then
\[k=(1-p)k_{0}+pk_{1}, \ \text{where}\ p=\frac{1+s-\eta}{1+s}. \]
Using the Sobolev embedding and the interpolation theorem, we have 
\begin{equation}
\begin{aligned}
\| q_{1}-q_{2}\|_{L^{\infty}(\Omega)} \leq C \|q_{1}-q_{2} \|_{H^{\frac{n}{2}+\eta}(\Omega)}&\leq C \|q_{1}-q_{2} \|_{H^{-1}(\Omega)}^{1-p} \cdot \|q_{1}-q_{2} \|_{H^{s}(\Omega)}^{p} \\
&\leq C \|q_{1}-q_{2} \|_{H^{-1}(\Omega)}^{\frac{\eta}{1+s}}\\
&\leq C \left[\omega^7 \|\wt{\Lambda}_{q_{1}}-\wt{\Lambda}_{q_{2}} \|+\frac{1}{\left[\frac{1}{K} \ln \left(\ln \ \omega+\vert \ln\ \|\wt{\Lambda}_{q_{1}}-\wt{\Lambda}_{q_{2}} \| \vert \right) \right]^{\frac{2}{\theta}}} \right]^{\frac{\theta \eta}{2(1+s)}},
\end{aligned}
\label{est-115}    
\end{equation}
where $C$ is a generic constant that does not depend on $\omega $. This gives the required stability estimate.

\section*{Acknowledgments }
We thank Arpit Babbar for the discussions regarding this problem. The research of the second author was supported by SERB Matrics grant MTR/2017/000837. }

\end{document}